\documentclass[a4paper,10pt,reqno]{amsart}

\usepackage{amsmath}
\usepackage{amsthm}
\usepackage{amssymb}
\usepackage{graphicx}
\usepackage{multicol}
\usepackage{appendix}
\usepackage[pdftex]{hyperref}
\hypersetup{breaklinks=true,colorlinks=true,citecolor=green,urlcolor=red,linkcolor=blue,pdfpagemode=UseNone}

\theoremstyle{plain}

\newtheorem{theorem}{Theorem}
\newtheorem{lemma}[theorem]{Lemma}
\newtheorem{corollary}[theorem]{Corollary}

\def\COMMENT#1{}
\let\COMMENT=\footnote

\begin{document}

\title{A fast algorithm for the gas station problem}
\author{Kleitos Papadopoulos}
\thanks{Kleitos Papadopoulos, InSPIRE, Agamemnonos 20, Nicosia, 1041, Cyprus, \href{mailto:kleitospa@gmail.com}{\tt kleitospa@gmail.com}}
\author{Demetres Christofides}
\thanks{Demetres Christofides, School of Sciences, UCLan, 7080 Pyla, Larnaka, Cyprus, \href{mailto:dchristofides@uclan.ac.uk}{\tt dchristofides@uclan.ac.uk}}
\date{\today}
\begin{abstract}
In the gas station problem we want to find the cheapest path between two vertices of an $n$-vertex graph. Our car has a specific fuel capacity and at each vertex we can fill our car with gas, with the fuel cost depending on the vertex. Furthermore, we are allowed at most $\Delta$ stops for refuelling. 

In this short paper we provide an algorithm solving the problem in $O(\Delta n^2 + n^2\log{n})$ steps improving an earlier result by Khuller, Malekian and Mestre.


\end{abstract}

\maketitle

\begin{multicols}{2}
\section{Introduction}

There are numerous problems in the literature in which the task is to optimise the travel from one location to another or to optimise a tour visiting a specific set of locations. The problems usually differ in the restrictions that we may put into the way we can travel as well as in the notion of what an `optimal' route means. Of course one could theoretically check all possible ways to travel and pick out the optimal one. However we care about finding the optimal route in a much quicker way as usually checking all possibilities is impractical.

One of the most widely known abstractions of travel optimisation problems is that of the shortest paths which although very general in their definition, fail to take into consideration most of the aspects that arise in real world. Perhaps one of their more practical generalisations is the `gas station problem', introduced by Khuller, Malekian and Mestre in~\cite{KMM}. Out of the infinitude of possible of possible parameters it includes one more central aspect of the travelling agent, that of its limited fuel capacity and fuel consumption during the travelling. As it is the case for the shortest paths problem its also with the gas station problem that it can be used in a variety of problems that are not directly related with travelling optimisation

The setting of the gas station problem is as follows:

We are given a complete graph $G=(V,E)$, two specific vertices $s,t$ of $G$ and functions $d:E \to \mathbb{R}^+$ and $c:V \to \mathbb{R}^+$. Finally we are also given positive numbers $U$ and $\Delta$.

Each vertex $v$ of $G$ corresponds to a gas station and the number $c(v)$ corresponds to the cost of the fuel at this station. Given an edge $e=uv$ of $G$, the number $d(e)=d(u,v)$ corresponds to the distance between the vertices $u$ and $v$, or what is essentially equivalent, to the amount of gas needed to travel between $u$ and $v$. Finally, the number $U$ corresponds to the maximum gas capacity of our car.

Our task is to find the cheapest way possible to move from vertex $s$ to vertex $t$ if we are allowed to make at most $\Delta$ refill stops.

We make two further assumptions:

Our first assumption concerns the function $d:E \to \mathbb{R}^+$. We will follow the natural assumption that it satisfies the triangle inequalities. I.e.
\[ d(u,v) + d(v,w) \geqslant d(u,w) \quad \text{for every $u,v,w \in V$.}\] 

Our second assumption concerns the amount of fuel that we have initially in our car. We will make the assumption that we start with an empty fuel tank. We also consider the filling in our tank at this vertex as one of the refill stops. This does not make much difference. Indeed suppose that initially we have an amount $g$ of gas. Instead of solving the gas station problem for the graph $G$, we modify this graph by adding a new vertex $s'$ . We define the new distances by $d(s',v) = U-g + d(s,v)$ for every $v \in V(G)$. I.e.\ the vertex $s'$ has distance $U-g$ from $s$ and furthermore the shortest path from $s'$ to any other vertex $v$ of $G$ is via $s$. We also define $c(s') = 0$. It is then obvious that if we start from $s'$, we should completely fill our tank and then move to vertex $s$. The only difference is that we are now allowed one fewer refill stop than before. So solving the gas station problem for $G$ starting from $s$ with $g$ units of gas is equivalent to solving the gas station problem for $G'$ starting from $s'$ with no gas. 

An algorithm solving the gas station problem that runs in $O(\Delta n^2 \log{n})$ was introduced by Khuller, Malekian and Mestre in~\cite{KMM}. The main result of our article is the following:

\begin{theorem}\label{MainTheorem}
Given an $n$ vertex graph $G$, there is an algorithm which solves the gas station problem with $\Delta$ stops in at most $O(\Delta n^2 + n^2 \log{n})$ steps.
\end{theorem}

We should point out that the algorithm in~\cite{KMM} makes similar assumptions to ours. It explicitly mentions the assumption that the car starts with an empty fuel tank. It does not mention explicitly the assumption that the distances in the graph need to satisfy the triangle inequalities. However it does use it implicitly in its Lemma~1. (See the proof of our Lemma~\ref{KMM Lemma} which makes explicit why we do indeed need the distances to satisfy the triangle inequalities.)

We will prove Theorem~\ref{MainTheorem} in the next section. In a couple of instances our algorithm will call some familiar algorithms with known running time. The interested reader can find more details about those algorithms in many algorithms or combinatorial optimisation books, for example in~\cite{CLRS}. 

\section{Proof of Theorem~\ref{MainTheorem}}

We start by ordering all edge distances. Since there are $O(n^2)$ edges, this can be done in $O(n^2 \log{n})$ steps, using e.g.\ heapsort. In fact we will not use the full ordering of the edge distances. What we will need are the following local orderings:

For each $v \in V$ we create an ordering $v_1,\ldots,v_{n-1}$ of the vertices of $V \setminus \{v\}$ such that $d(v,v_i) \leqslant d(v,v_j)$ for $i \leqslant j$. We will call this the {\bf local edge ordering} at $v$. (Note that this definition might be a bit misleading as the local edge ordering at $v$ is an ordering of the vertices of $V \setminus \{v\}$. Of course, this gives an ordering of the edges incident to $v$ and this is where it gets its name from.)

Of course all of these local orderings can also be computed in $O(n^2 \log{n})$ steps.

So it is enough to show how to solve the gas station problem in $O(\Delta n^2)$ time assuming that the edges are already ordered by distance. 

The fact that the edge distances satisfy the triangle inequalities is needed to prove the following simple lemma:

\begin{lemma}\label{KMM Pre-Lemma}
There is an optimal route during which we fill our car with a positive amount of gas at every station (apart from the last one). 
\end{lemma}

\begin{proof}
Amongst all optimal routes, pick one passing through the smallest number of vertices. Suppose that it passes through vertices $v_1,v_2,\ldots,v_k$ in that order. We definitely need to fill our car with gas at $v_1$ as we start with an empty tank. Suppose now for contradiction that in this optimal path we do not fill our car at station $v_i$ for some $1 < i < k$. Then, instead of moving from $v_{i-1}$ to $v_{i+1}$ through $v_i$, we could have moved to it directly. This is indeed possible as 
\[d(v_{i-1},v_{i+1}) \leqslant d(v_{i-1},v_i)+d(v_i,v_{i+1})\]
and it is a contradiction as we assumed that our optimal path is both optimal and minimal.
\end{proof}

Lemma~\ref{KMM Pre-Lemma} is needed to prove the following slightly modified lemma from~\cite{KMM}. Even though it looks completely obvious, we nevertheless provide a detailed proof in order to make explicit the need for using Lemma~\ref{KMM Pre-Lemma} and thus to require that the distances in the graph satisfy the triangle inequalities. Lin~\cite{Lin} also makes explicit this requirement. Essentially the same lemma also appears in~\cite{LGR}.

\begin{lemma}\label{KMM Lemma} There is an optimal route, say passing through vertices $v_1,v_2,\ldots,v_k$ in that order, where $v_1 = s$ and $v_k = t$, for which an optimal way to refill the tank is as follows:
\begin{itemize}
\item[(i)] For $1 \leqslant i \leqslant k-2$, if $c(v_i) < c(v_{i+1})$, then at station $v_i$ we completely fill our tank.
\item[(ii)] For $i=k-1$, if $c(v_i) < c(v_{i+1})$, then at station $v_i$ we fill the tank with just enough gas in order to reach vertex $v_{i+1}$ with an empty tank.
\item[(iii)] For $1 \leqslant i \leqslant k-1$, if $c(v_i) \geqslant c(v_{i+1})$ then at station $v_i$ we fill the tank with just enough gas in order to reach vertex $v_{i+1}$ with an empty tank.
\end{itemize}
\end{lemma}

\begin{proof}
Pick an optimal route as given by Lemma~\ref{KMM Pre-Lemma}.

Suppose that at some point during travelling through this optimal route we reach vertex $v_i$, with $1 \leqslant i \leqslant k-2$ and suppose that $c(v_i) < c(v_{i+1})$. By Lemma~\ref{KMM Pre-Lemma} we have filled some gas at station $v_{i+1}$. If we did not fully filled our gas at station $v_i$ then we could have reduced our cost by filling more gas at $v_i$ and less at $v_{i+1}$, a contradiction. So at $v_i$ we definitely must completely fill our tank. 

If we reach $v_{k-1}$ then of course we fill the car with just enough gas in order to reach $v_k$ with our tank completely empty.

Finally suppose that at some point during travelling through this optimal route we reach vertex $v_i$, with $1 \leqslant i \leqslant k-1$ and $c(v_i) \geqslant c(v_{i+1})$. Suppose we fill the car at $v_i$ with $a$ units of gas and that when reaching $v_{i+1}$ we still have $b > 0$ units of gas left. Suppose that at $v_{i+1}$ we fill our tank with $c > 0$ units of gas. Of course we could still  reach $v_{i+1}$ with the same amount of gas but spend at most as much as before if we filled our car with $a-b$ units of gas at $v_i$ and $b+c$ units of gas at $v_{i+1}$.

So indeed the process described in the Lemma is an optimal refilling process. 
\end{proof}

An immediate and important Corollary of Lemma~\ref{KMM Lemma} is the following:

\begin{corollary}\label{Cor}
There is an optimal route such that for every vertex $u$ that we reach, either at the previous step we had an empty tank, or two steps before we completely filled our tank.
\end{corollary}

Given a vertex $u \in V$ and an integer $r \in \{0,1,\ldots,\Delta\}$ we write $C_0[u,r]$ for the minimal cost of starting from $s$ with an empty tank, and finishing at $u$ with an empty tank going through exactly $r$ edges. We also write $C_U[u,r]$ for the minimal cost of starting from $s$ with an empty tank, and finishing at $u$ with a full tank going through exactly $r$ edges.

If it is impossible to be on the vertex $u$ going through exactly $r$ edges then we define $C_0[u,r]=C_U[u,r] = \infty$.

The use of the symbol $\infty$ is purely for convenience of the proof. We could have just said that $C_0[u,r]$ and $C_U[u,r]$ are undefined. We will in fact make further use of this symbol. Note that given two vertices $u,v$ with $d(u,v) > U$ we can never travel between $u$ and $v$. Rather than separating the cases in which $d(u,v) > U$ or not it will be simpler to make all these distances equal to $\infty$. To be more explicit, we define a new function $d':E \to \mathbb{R}^+ \cup \{\infty\}$ by 
\[ d'(e) = \begin{cases} d(e) & \text{if } d(e) \leqslant U \\ \infty & \text{otherwise.}\end{cases}\]
In what follows, we might often travel through edges $e$ with $d'(e) = \infty$. The total cost of such a travel will be equal to $\infty$. At the end of the algorithm the minimal cost computed will be less than $\infty$, guaranteeing that the optimal route produced does not travel through such edges. (Unless of course it is impossible to make such a travel, whatever the cost.)

Note that the ordering of the edges we have already computed using the function $d$ is also an ordering of the edges using the function $d'$. On the other hand $d'$ does not satisfy the triangle inequalities but we need not worry about this as we have already used them for  Corollary~\ref{Cor}. Of course changing $d$ to $d'$ does not alter the conclusion of the Corollary even though $d'$ does not satisfy the triangle inequalities.

Our task is to calculate all $C_0[t,r]$ for $0 \leqslant r \leqslant \Delta$. The required quantity will then be the minimum of these. This minimum can be calculated in $\Delta$ steps. 

We will proceed inductively as follows:

Step $2r+1$: We calculate all $C_0[u,r]$ with $u \in U$ in $O(n^2)$ time.

Step $2r+2$: We calculate all $C_U[u,r]$ with $u \in U$ in $O(n^2)$ time.

We can stop at Step $2\Delta+1$ and so the total running time will be $O(\Delta n^2)$ as required.

The first two steps are easy to perform and can even be calculated in $O(n)$ time as we have
\[ C_0[u,0] = \begin{cases} 0 & u=s \\ \infty & u \neq s  \end{cases}\]
and
\[ C_U[u,0] = C_0[u,0] + Uc(u).\]

So now we assume that $r \geqslant 1$. For ease of exposition only, we will first explain how to perform Step $2r+2$ as is is easier to explain. To perform Step $2r+2$ we will assume that we already performed Step $2r+1$. Of course there is no problem with this as long as when we explain how to perform Step $2r+1$ we do not make any use of the Step $2r+2$.

For Step $2r+2$ fix a specific $u \in U$. As there are $n$ such $u$'s, it is enough to show how to compute $C_U[u,r]$ in $O(n)$ time.

Note that from Corollary~\ref{Cor}, if $u$ is the $r$-th vertex that we reach, either we reach it with an empty tank, or we leave the $(r-1)$-th vertex with a full tank.

If the first case happens then 
\[C_U[u,r] = C_0[u,r] + Uc(u).\] 
If the second case happens and $v$ is the $(r-1)$-th vertex, then 
\[C_U[u,r] = C_U[v,r-1] + d'(v,u)c(u).\]
So $C_U[u,r]$ is the minimum of $n$ quantities and can therefore be calculated in $O(n)$ time as required. [Note that in the second case we should have considered only those $v$ for which $d(v,u) \leqslant U$. Recall however that for all the other vertices we defined $d'(v,u) = \infty$ and so the total cost through such vertices is $\infty$. The optimal cost is finite and therefore the optimal route does not pass through such vertices unless of course there is no possible way to reach vertex $u$ through $r$ edges.]

We now explain how to perform Step $2r+1$. We will in fact consider many vertices at once, but for the moment fix a specific $u \in U$ and suppose that it is the $r$-th vertex that we reach and we reach it with an empty tank. From Corollary~\ref{Cor} there are two possible ways this could happen:
\begin{itemize}
\item[Type I:] We reach the $(r-1)$-st vertex with an empty tank. 
\item[Type II:] We leave the $(r-2)$-nd vertex with a full tank.
\end{itemize}

The main difficulty is to treat the `Type II' possibilities as a naive way to do it needs $\Theta(n^2)$ time for each $u \in V$ and thus $\Theta(n^3)$ time for all $u \in V$.

We write $C'_0[u,r]$ for the minimum cost over all `Type I' ways in which we can reach $u$ with an empty tank, and $C''_0[u,r]$ for the minimum over all `Type II' ways.

Note that if we reach $u$ with a `Type I' way through a vertex $v$ then 
\[ C_0[u,r] = C_0[v,r-1] + d'(u,v)c(v)\] 
So we can compute $C'_0[u,r]$ in $O(n)$ time. Thus we can compute $C'_0[u,r]$ for all $u \in V$ in $O(n^2)$ time.

It remains to compute $C''_0[u,r]$. If we can compute it for all $u \in V$ in $O(n^2)$ time then we will be done. Indeed, then for each $u \in V$ we have
\[ C_0[u,r] = \min\{C'_0[u,r],C''_0[u,r]\}\]
and we can compute this in $O(n)$ time for all $u \in V$. 

To compute $C''_0[u,r]$ we introduce yet another piece of notation and define $C''_0[u,r;v]$ as the minimum over all `Type II' ways such that the $(r-1)$-st vertex is $v$. Evidently,
\[ C''_0[u,r] = \min_{v \in V \setminus \{u\}}C''_0[u,r;v]\]

For each fixed $v$ we will compute $C''_0[u,r;v]$ for all $u \neq v$ in $O(n)$ time. This will be enough for our purposes. Indeed this means that we can compute $C''_0[u,r;v]$ for all $u,v \in V$ in $O(n^2)$ time. But then for each fixed $u$ we can compute $C''_0[u,r]$ in an additional $O(n)$ time and thus for all $u$ in an additional $O(n^2)$ time.

To do this observe that
\begin{multline*}
C''_0[u,r;v] = \\ \min_{w \in S(u,v)} \left[C_U[w,r-2] + (d'(w,v)+d'(v,u)-U)c(v) \right]
\end{multline*}
where $S(u,v)$ is the set of all vertices $w$ for which such `Type II' ways are possible. I.e.\ all $w \in V \setminus \{v\}$ such that $d'(w,v) + d'(v,u) \geqslant U$ or equivalently all $w \in V \setminus \{v\}$ such that $d(w,v) + d(v,u) \geqslant U$ 

The two main observations that will allow us to compute all $C''_0[u,r;v]$ fast enough are the following:

Firstly, we can rewrite
\begin{multline}\tag{$\ast$} 
C''_0[u,r;v] = d'(v,u)c(v) + \phantom{} \\ \min_{w \in S(u,v)} \left[C_U[w,r-2] - (U-d'(w,v))c(v)\right]  
\end{multline}
To see the importance of this observation, suppose for a moment that $S(u,v) = V \setminus \{v\}$ for each $u,v \in V$. (We will deal with the general situation in our second observation.) Then we could compute all
\[ \min_{w \in S(u,v)} \left[C_U[w,r-2] - (U-d'(w,v))c(v)\right] \]
in $O(n^2)$ time by taking $O(n)$ time for each $v$. But then we could compute all $C''_0[u,r;v]$ in $O(n^2)$ time as required.

In principle however, the sets $S(u,v)$ can be a lot different. However, for $u,u'$ with $d(u,v) \leqslant d(u',v)$ we have $S(u,v) \subseteq S(u',v)$. Indeed if $w \in S(u,v)$, then  
\[d(w,v) + d(v,u') \geqslant d(w,v) + d(v,u) \geqslant U  \]
so we also have $w \in S(u',v)$.

This second observation says that there is some structure to the sets $S(u,v)$. We will use it in order to still manage to compute all 
\[ \min_{w \in S(u,v)} \left[C_U[w,r-2] - (U-d'(w,v))c(v)\right] \]
in $O(n^2)$ time by taking $O(n)$ time for each $v$.

To prepare for its use we need the following lemma:
\begin{lemma}\label{local-lists}
For each fixed $v \in V$ and let $v_1,v_2,\ldots,v_{n-1}$ be the local edge ordering at $v$. Define $T(v_1,v) = S(v_1,v)$ and $T(v_i,v) = S(v_i,v) \setminus S(v_{i-1},v)$ for $2\leqslant i \leqslant n-1$. Then we can determine all $T(v_i,v)$ in $O(n)$ time. 
\end{lemma}

\begin{proof}
Starting from the last vertex, we go backwards through all the vertices of the local edge ordering at $v$ until we find one which does not belong to $S(v_1,v)$. Suppose $v_{k_1}$ is this vertex. From our second observation we have that 
\[ T(v_1,v) = S(v_1,v) = \{v_{k_1+1},\ldots,v_n\}.\]
In the case that $v_{k_1}=v_n$, the understanding is that $T(v_1,v) = \emptyset$. 

Now we start from $v_{k_1}$ and we again go backwards through all the vertices of the local edge ordering until we find one which does not belong to $S(v_2,v)$. Suppose $v_{k_2}$ is this vertex. So from our second observation we have that 
\[ S(v_2,v) = \{v_{k_2+1},\ldots,v_n\}\]
and
\[ T(v_2,v) = \{v_{k_2+1},\ldots,v_{k_1}\}.\]
Continuing inductively we can find $k_1,k_2,\ldots,k_{n-1}$ where $k_i$ is the first vertex in the local edge ordering, starting from $k_{i-1}$ and going backwards, which does not belong to $S(v_i,v)$. Then
\[ S(v_i,v) = \{v_{k_i+1},\ldots,v_n\}\]
and
\[ T(v_i,v) = \{v_{k_i+1},\ldots,v_{k_{i-1}}\}.\]
Note that at each step of the process we ask a question of the form `Does vertex $v_i$ belong to the set $S(v_j,v)$?'. This question is answered in $O(1)$ time by checking whether a specific inequality holds. If needed, we then add $v_i$ to the set $T(v_j,v)$ in $O(1)$ time. 

Furthermore, we have at most $n$ questions from which we obtain a `Yes' answer, as each time we obtain a `Yes' we move on to the next vertex in the local edge ordering. We also have at most $n$ questions from which we obtain a `No', as each time we obtain a `No', we move on to examining inclusion in the next set of the form $S(v_j,v)$. So all these checks and additions to lists can be done in $O(n)$ time as required. 
\end{proof}

Now we can proceed calculating $C''_0[u,r;v]$ for all $u \in V$ as follows:

Consider the local edge ordering $v_1,v_2,\ldots,v_{n-1}$ at $v$. We will show that for each $1 \leqslant m \leqslant n$ we can compute
\[ C''_0[v_1,r;v],\ldots,C''_0[v_m,r;v]\]
in $O(|S(v_m,v)|)$ time. 

Note that this completes the proof of Theorem~\ref{MainTheorem}. Indeed by taking $m=n-1$ we get that we can for fixed $v \in V$ compute all $C''_0[u,r;v]$ in $O(n)$ time for all $u\in V$ as required.

We will prove our claim by induction on $m$. This is immediate for $m=1$ since from $(\ast)$ we just need to compute the minimum over all elements of $T(v_1,v) = S(v_1,v)$. Suppose now that it is true for $m=k$. Then it is also true for $m = k+1$. Indeed again from our second observation we have $S(v_{k+1},v) = S(v_k,v) \cup T(v_{k+1},v)$.  Thus to compute the minimum of 
\[C_U[w,r-2] - (U-d'(w,v))c(v)\] 
over all $w \in S(v_{k+1},v)$ we can start with the minimum over all $w \in S(v_k,v)$ which is already computed and additionally consider the minimum over all elements of $T(v_{k+1},v)$. This can be done in another $O(|T(v_{k+1},v)|)$ steps and since the 
sets $S(v_k,v)$ and $T(v_{k+1},v)$ are disjoint, the total time taken so far is $O(|S(v_{k+1},v)|)$. So the claim is also true for $m=k+1$ and this completes the proof of the claim.

\section{Generalisations}

We finally conclude with a list of comments mostly concerned with possible generalisations of the algorithm.
\begin{enumerate}
\item The algorithm can also treat non-complete graphs as well. For every edge $uv$ that does not belong to $G$ we define $d(u,v)=U+1$. Thus we will never pass through that edge and so this is the same as when the edge does not exist. In that case though, the edge distances might not satisfy the triangle inequalities and so we would have to compute all shortest paths as in the previous generalisation. (If for each edge $uv$ we replace its edge length by the lengths of the shortest path between $u$ and $v$ then the new graph satisfies the triangle inequalities.) 
\item Given a graph with $n$ vertices and $m$ edges, by using Johnson's algorithm the shortest paths can be computed in $O(n^2\log{n}+mn)$ time. So in particular the total running time becomes $O(n^2\log{n} + \Delta n^2 + mn)$.
\item It would be interesting to solve to gas station problem for particular classes of graphs. For example~\cite{LGR} gives a linear time algorithm in the case that the graph is a path. 
\item The algorithm also treats directed graphs as well. Here, given two vertices $u$ and $v$ we allow $d(u,v) \neq d(v,u)$. This is relevant in various situations. For example going up the mountains needs more fuel consumption than returning back from the mountains via the same road. 
\item If $\Delta = n$ our algorithm solves the gas station problem in $O(n^3)$ time. The same time is also achieved in~\cite{KMM} and in~\cite{Lin}. (In fact in~\cite{Lin} this is even achieved for all pairs of vertices at once.)  
\item We have actually solved the gas station problem not only for the travel from the fixed starting vertex $s$ to the fixed terminal vertex $t$, but actually for the travel from the fixed starting vertex $s$ to every other vertex $u$ of $V$.
\item By applying the algorithm separately for each possible starting vertex (note that the local edge orderings are only computed once) we get a $O(\Delta n^3)$ algorithm for the all-pairs version of the gas station problem. This is better than the $O(\Delta^2 n^3)$ algorithm in~\cite{KMM} but not as good as the $O( n^3 \log{\Delta})$ algorithm by Lin in~\cite{Lin}.  
\item In practice, we do not only care about the minimum cost, but also on the minimum time needed for the travel. This was actually the logic behind the restriction of having at most $\Delta$ stops. This can still be treated by our algorithm if we have a `cost function' for the travel between each pair of vertices and the `total cost' is the sum of the individual costs.
\item The space complexity of the algorithm is $O(n^2)$. Indeed note that in performing Steps $2r+1$ and $2r+2$ we only need the knowledge of the following: The $n$ local edge orderings, with each local edge ordering having space complexity $O(n)$, the $\binom{n}{2}$ edge distances, $O(n)$ values of the form $C_0[u,r-1], C_U[u,r-1], C_0[u,r]$ and $C_U[u,r]$, where some of those values are known from before and some of those are calculated during the algorithm. Finally we also need $O(n^2)$ values of the form $C_0''[u,r;v]$. Note in particular than once we are done with Steps $2r+1$ and $2r+2$ we do not use the $C_0''[u,r;v]$ anymore.
\end{enumerate}

%
%

\end{multicols}


\begin{thebibliography}{9}
\bibitem{CLRS} T. H. Cormen, C. E. Leiserson, R. L. Rivest\ and\ C. Stein, {\it Introduction to algorithms}, third edition, MIT Press, Cambridge, MA, 2009.

\bibitem{KMM} S. Khuller, A. Malekian\ and\ J. Mestre, To fill or not to fill: the gas station problem, ACM Trans. Algorithms {\bf 7} (2011), Art. 36, 16 pp. 

\bibitem{Lin} S. H. Lin, Finding optimal refueling policies in transportation networks, in {\it Algorithmic Aspects in Information and Management}, Lecture Notes in Computer Sciences {\bf 5034} (2008), 280--291. 

\bibitem{LGR} S. H. Lin, N. Gertsch\ and\ J. R. Russell, A linear-time algorithm for finding optimal vehicle refueling policies, Oper. Res. Lett. {\bf 35} (2007), 290--296.
\end{thebibliography}
\end{document}